\documentclass[11pt]{article}

\usepackage{amsfonts}
\usepackage{amsmath}
\usepackage{color}

\newcommand{\beq}{\begin{equation}}
\newcommand{\eeq}{\end{equation}}
\newcommand{\bea}{\begin{eqnarray}}
\newcommand{\eea}{\end{eqnarray}}
\newcommand{\beas}{\begin{eqnarray*}}
\newcommand{\eeas}{\end{eqnarray*}}

\newtheorem{theorem}{Theorem}[section]

\newtheorem{proposition}[theorem]{Proposition}
\newtheorem{prop}[theorem]{Proposition}

\newtheorem{lemma}[theorem]{Lemma}
\newtheorem{remark}[theorem]{Remark}
\newtheorem{example}[theorem]{Example}
\newtheorem{examples}[theorem]{Examples}
\newtheorem{foo}[theorem]{Remarks}

\newenvironment{proof}{\addvspace{\medskipamount}\par\noindent{\it
Proof}.}
{\unskip\nobreak\hfill$\Box$\par\addvspace{\medskipamount}}

\newcommand{\bH}{\mathbb H}
\newcommand{\bS}{\mathbb S}

\newcommand{\R}{\mathbb R}

\title{The subelliptic heat kernel on the anti-de Sitter space}
\author{Jing Wang\footnote{wangj@ima.umn.edu} }

\date{IMA, University of Minnesota\\
Minneapolis, MN, USA}

\begin{document}
\maketitle

\begin{abstract}
We study the subelliptic heat kernel of the sub-Laplacian on a $2n+1$-dimensional anti-de Sitter space $\bH^{2n+1}$ which  also appears as a model space of a CR Sasakian manifold with constant negative sectional curvature. In particular we obtain an explicit and geometrically meaningful formula for the subelliptic heat kernel. The key idea is to work in a set of coordinates  that reflects the symmetry coming from the Hopf fibration  $\mathbb{S}^1\to \bH^{2n+1}$. A direct application is obtaining small time asymptotics of the subelliptic heat kernel. Also we derive an explicit formula for the sub-Riemannian distance on $\bH^{2n+1}$.
\end{abstract}

\tableofcontents
\clearpage

\section{Introduction}

As one of the closest relatives of Minkowski space-time, the anti-de Sitter space is of great interest and importance in theoretical physics. As the solution of Einstein's equation with an attractive cosmological constant and of maximal symmetry, it plays an essential role in the study of relativity and BTZ black holes (see \cite{Bengtsson2},  \cite{Ca}, \cite{Na} and references therein).

From the  geometric point of view, an anti-de Sitter space is especially interesting for it appears as an example in both sub-Lorentzian geometry and sub-Riemannian geometry (see \cite{CMV}). In this paper, we are particularly interested in its sub-Riemannian structure, with which the anti-de Sitter space appears as a model space of CR Sasakian manifold with constant negative sectional curvature. We study the subelliptic diffusion that is canonically associated to this sub-Riemannian structure. In particular, we give an explicit expression of the subelliptic heat kernel in terms of the Riemannian heat kernel of a  hyperbolic space. 
From this integral representation, we derive the small time asymptotic estimations of the  kernel, and obtain three different behaviors: on the diagonal points, on the vertical cut-locus, and outside of the cut-locus.%, which illustrate the sub-Riemannian structure.

Throughout this paper, we denote the $2n+1$-dimensional anti-de Sitter space by $\bH^{2n+1}$ and the $2n+1$-dimensional hyperbolic space by $\mathbf{H}^{2n+1}$. 
The key idea is to take advantage of the strong geometric symmetry that comes from the Hopf fibration:
\[
\bS^1\longrightarrow\bH^{2n+1}\longrightarrow\mathbb{CH}^n
\] 
where $\mathbb{CH}^n$ is the complex hyperbolic space of dimension $2n$.  The indefinite Riemannian submersion $\bH^{2n+1}\rightarrow\mathbb{CH}^n$ is compatible with the contact structure on $\bH^{2n+1}$. Here, the Reeb vector field is given by the generator $T$ of the group action $\bS^1$ on $\bH^{2n+1}$, and the kernel of the contact form (pseudo-Hermitian form) is the horizontal distribution  of the submersion.
As a  model space of CR Sasakian manifold, $\bH^{2n+1}$ is transversely symmetric (namely its pseudo-Hermitian torsion vanishes). The analytic interpretation  is that Reeb vector field $T$ commutes with the sub-Laplacian $L$ on $\bH^{2n+1}$:
\[
LT=TL.
\] 
Another important observation is the analytic continuation of hyperbolic space $\mathbf{H}^{2n+1}$ to $\bH^{2n+1}$, which leads to the relation between the sub-Laplacian $L$ on $\bH^{2n+1}$ and the Riemannian Laplacian $\square$ on $\mathbf{H}^{2n+1}$:
\[
L=\square+T^2.
\] 
With these two ingredients we have heuristically $e^{tL}=e^{t\square}e^{tT^2}$, which suggests the integral representation of the subelliptic heat kernel.  Noting that $\bH^{2n+1}$ is not simply connected, we first work on its universal covering $\widetilde{\bH^{2n+1}}$ with fiber $\R$, then obtain the results on $\bH^{2n+1}$ by ``wrapping'' the ones on $\widetilde{\bH^{2n+1}}$.
We  now state the main result (see Proposition \ref{prop1}). Let $p_t^{\bH^{2n+1}}$ be the subelliptic heat kernel of $\bH^{2n+1}$ and $q_t$ the Riemannian heat kernel of  $\mathbf{H}^{2n+1}$.
\begin{itemize}
\item For $t>0$, $r\in[0,+\infty)$, $ \theta\in(-\infty,+\infty)$, the subelliptic heat kernel on the universal covering of the anti-de Sitter space $\widetilde{\bH^{2n+1}}$ is given by
\[
p_t^{\widetilde{\bH^{2n+1}}}(r, \theta)=\frac{1}{\sqrt{4\pi t}}\int_{-\infty}^{+\infty}e^{\frac{(y-i \theta)^2}{4t} }q_t(\cosh r\cosh y)dy.
\]
\item For $t>0$, $r\in[0,+\infty)$, $ \theta\in[-\pi,\pi]$, the subelliptic heat kernel on  $\bH^{2n+1}$ is given by
\[
p_t^{\bH^{2n+1}}(r, \theta)=\frac{\Gamma(n+1)e^{-n^2t+\frac{\pi^2}{4t}}}{(2\pi)^{n+2} t}\sum_{k\in\mathbb{Z}}\int_{-\infty}^{+\infty}\int_0^{+\infty}\frac{e^{\frac{(y-i\theta-2k\pi i)^2-u^2}{4t}}\sinh u\sin\left(\frac{\pi u}{2t}\right)}{\left(\cosh u+\cosh r\cosh y\right)^{n+1}}du
dy.
\]
\end{itemize}

As an application, we  derive the small time asymptotics of the subelliptic heat kernel by applying the steepest descent method. We shall observe three different structure-illustrating behaviors of the small time asymptotics: on  diagonal points, on the vertical cut-locus, and outside the cut-locus. Another by-product is the explicit formula for the sub-Riemannian distance on $\bH^{2n+1}$.

The study of finding explicit formulas for subelliptic heat kernels has generated a great amount of work (see \cite{A}, \cite{B},  \cite{BB},  \cite{BW}, \cite{BGG}, \cite{Bo}, \cite{Ga}  and the references therein). 
The pioneering work is due to Gaveau (see \cite{Ga}), who obtained an integral representation of the subelliptic heat kernel on a Heisenberg group. A Heisenberg group appears as a flat model of a sub-Riemannian manifold. As for the non-flat models, Baudoin-Bonnefont (see \cite{BB}) and Bonnefont (see \cite{Bo}) derived the subelliptic heat kernel on $\rm{SU}(2)$ and $\rm{SL}(2,\R)$, which appear respectively as $3$-dimensional sub-Riemannian model spaces with constant positive and negative curvatures. %However, despite such numerous works, few explicit and tractable formulas are actually known and most of them are restricted to a Lie group framework. 
In the case of higher dimensional models, the explicit subelliptic heat kernels on the CR sphere $\bS^{2n+1}$ (where group structure no longer exists) were first deduced in \cite{BW}. An Hamiltonian approach to the heat kernel was then given by Greiner in \cite{Greiner}.  In present work we discuss the hyperbolic counterpart of the models in \cite{Ga} and \cite{BW}, by giving explicit and geometrically meaningful expressions that hold  in a natural  sub-Riemannian model with negative sectional curvature. 

Explicit expressions of heat kernels have  varied applications, such as determination of sharp constant in functional inequalities  (\cite{BFM}, \cite{LF}),  computation of the sub-Riemannian metric (\cite{BB}), sharp upper and lower bounds for the heat kernel (\cite{BGG},\cite{El}),  and semigroup sub-commutations (\cite{Li}). In this article, the most important application of the integral representation is to derive small time asymptotic estimations of the subelliptic heat kernel on $\bH^{2n+1}$. 

The results of small time estimates for the heat kernels that are associated with general hypoelliptic diffusion operators were given by L\'eandre (\cite{Leandre1}, \cite{Leandre2}) and Ben Arous (\cite{GBA1}, \cite{GBA2}) in the 1980's using probabilistic approaches.  Beals-Greiner-Stanton  studied the small time asymptotics of  subelliptic heat kernels on CR manifolds by using pseudo-differential calculus (see \cite{BGS}). Obtaining small-time estimates through  explicit expressions of the heat kernel began with Beals-Gaveau-Greiner on the Heisenberg group (see \cite{BGG}). In \cite{BB}, Baudoin-Bonnefont derived small time estimates on $\rm{SU}(2)$, and later Bonnefont gave the results on  $\rm{SL}(2,\R)$ and its universal covering (see \cite{Bo}). In this paper, by studying the asymptotic estimations on $\bH^{2n+1}$, we provide another  example of studying the limiting behaviors of the degenerated subelliptic  heat kernels, with the result on the cut-locus of particular interest. 

The paper is organized as follows. In the next section we study the geometry of the anti-de Sitter space $\bH^{2n+1}$ and deduce the sub-Laplacian that is associated to its sub-Riemannian structure. In section $3$ we deduce our main result: the integral representation  for the subelliptic heat kernels on $\bH^{2n+1}$ and  $\widetilde{\bH^{2n+1}}$. These expressions are precise enough to derive the small time asymptotics of the kernel. As a consequence, we obtain an expression for the sub-Riemannian distance.

\section{The sub-Laplacian on $\mathbb{H}^{2n+1}$ and $\widetilde{\bH^{2n+1}}$ }

\subsection{Hopf fibration and sub-Riemannian geometry on anti-de Sitter spaces}
As solutions of Einstein's equation with attractive cosmological constants and maximal symmetry, anti-de Sitter spaces play an important role in mathematical physics (see \cite{Ca}, \cite{Na}). 
A $2n+1$-dimensional anti-de Sitter $\bH^{2n+1}$ is defined as a quadratic 
\[
x_1^2+y_1^2+\cdots+x_n^2+y_n^2-x_{n+1}^2-y_{n+1}^2=-1
\]
embedded in a flat $2n+2$-dimensional space $\R^{2n,2}$ with the metric of  Lorentzian signature $(2n,2)$:
\begin{equation}\label{eq-lorent}
ds^2=dx_1^2+dy_1^2+\cdots+dx_n^2+dy_n^2-dx_{n+1}^2-dy_{n+1}^2.
\end{equation}
Notice that the ambient metric restricted to the tangent bundle $T(\bH^{2n+1})$ has Lorentzian signature $(2n,1)$. Precisely speaking, the target bundle contains $1$ time-like vector field and $2n$ space-like vector fields. Now let us consider a smooth $2n$-dimensional bracket generating distribution $\mathcal{H}$ on $T(\bH^{2n+1})$. There are the following two cases  representing two different geometries: 
\begin{itemize}
\item[(1)] If $\mathcal{H}$ is generated by $2n$ space-like vector fields, then $(\bH^{2n+1}, \mathcal{H},\langle\cdot, \cdot\rangle_{\mathcal{H}})$ is a sub-Riemannian manifold whose sub-Riemannian metric has signature $(2n,0)$.  $\mathcal{H}$ is then the so-called horizontal distribution.
\item[(2)] If $\mathcal{H}$ is generated by $2n-1$ space-like vector fields and a time-like vector field, then the triple $(\bH^{2n+1}, \mathcal{H},\langle\cdot, \cdot\rangle_{\mathcal{H}})$ is the so-call sub-Lorentzian manifold whose sub-Lorentzian metric has  signature $(2n-1,1)$. 
\end{itemize}
In \cite{CMV}, Chang-Markina-Vasil'ev studied both of these cases on $\bH^{3}$ (where the $3$-dimensional anti-de Sitter space is denoted by $\mathit{AdS}$ by the authors). In this article we are mostly interested in the sub-Riemannian structure of $\bH^{2n+1}$. As a sub-Riemannian manifold, $\bH^{2n+1}$ also appears as a model space of CR Sasakian manifold with constant negative sectional curvature. From now on we will work in complex coordinates.  We denote
\[
\|z\|_H^2=\sum_{k=1}^n|z_{k}|^2-|z_{n+1}|^2.
\]
Then $\bH^{2n+1}$ is given by 
\[
\bH^{2n+1}=\lbrace z=(z_1,\cdots,z_{n+1})\in \mathbb{C}^{n+1}, \| z \|^2_H =-1\rbrace.
\]
We consider the isometric group action $\mathbb{S}^1\to\bH^{2n+1}$ such that for $z\in \bH^{2n+1}$,
\begin{equation}\label{eq-action}
(z_1,\cdots, z_n) \rightarrow (e^{i\theta} z_1,\cdots, e^{i\theta} z_n). 
\end{equation}
The orbit $z_\theta=e^{i\theta}z$ satisfies $d z_\theta/d\theta=iz_\theta$ and hence lies in the negative definite plane spanned by $\{z, iz\}$. The identification space of this action is a complex hyperbolic space $\mathbb{CH}^n$, which has negative constant holomorphic sectional curvature (see \cite{magid}). $\bH^{2n+1}\to\mathbb{CH}^n$ is in fact an indefinite Riemannian submersion.

We denote the generator of this group action \eqref{eq-action} by $T$ throughout the paper. For any $f \in{C}^\infty(\bH^{2n+1})$,
\[
Tf(z)=\frac{d}{d\theta}f(e^{i\theta}z)\mid_{\theta=0},
\]
thus
\[
T=i\sum_{j=1}^{n+1}\left(z_j\frac{\partial}{\partial z_j}-\overline{z_j}\frac{\partial}{\partial \overline{z_j}}\right).
\]
Moreover, this Riemannian submersion is compatible with the contact structure  of $\bH^{2n+1}$ in the sense that the horizontal distribution of the submersion is in fact the kernel of the contact form.
% which, also known as the projective model of 
%Let $\eta$ be a pseudo-Hermitian contact form on  $\bH^{2n+1}$ that trivialize $T$, then
%\begin{equation}\label{eqeta}
%\eta=\frac{i}{2}\left(\sum_{j=1}^{n}(\overline{z_j}dz_j-z_j d\overline{z_j})-(\overline{z_{n+1}}dz_{n+1}-z_{n+1} d\overline{z_{n+1}})\right).
%\end{equation}
%The vector field $T$ is the Reeb vector field (characteristic direction) of $\eta$, i.e., $\eta(T)=1$, and we can easily find that $\eta=\frac{|z_{n+1}-1|^2}{2}\eta_1$.
We choose a contact form $\eta$ whose characteristic vector field is given by $-T$, i.e., $\eta(-T)=1$, $d\eta (-T,\cdot)=0$. Then in local coordinates we have 
\begin{equation}\label{eqeta}
\eta=-\frac{i}{2}\left(\sum_{j=1}^{n}(\overline{z_j}dz_j-z_j d\overline{z_j})-(\overline{z_{n+1}}dz_{n+1}-z_{n+1} d\overline{z_{n+1}})\right).
\end{equation}
The associated Levi form is given by 
\[
\mathcal{L}_\eta=\frac{1}{2}\sum_{k=1}^ndz_k\wedge d\overline{z_k}-dz_{n+1}\wedge d\overline{z_{n+1}}.
\]
The horizontal distribution of $\bH^{2n+1}$ is given by the Levi distribution $\mathcal{H}(\bH^{2n+1}):={\mathrm{Re}}\{T^{1,0}(\bH^{2n+1})\oplus T^{0,1}(\bH^{2n+1})\}$, and the sub-Riemannian structure $g_\eta$ induced from the Levi form is such that 
\[
g_\eta(X,Y)=(d\eta)(X, JY),\quad \mbox{for all $X, Y\in \mathcal{H}(\bH^{2n+1})$},
\]
where $J:\mathcal{H}(\bH^{2n+1})\to \mathcal{H}(\bH^{2n+1})$ is the complex structure given by $J(V+\overline{V})=i(V-\overline{V})$ for all $V\in T^{1,0}(\bH^{2n+1})$. In fact $\mathcal{L}_\eta$ coincides on $T^{1,0}(\bH^{2n+1})\oplus T^{0,1}(\bH^{2n+1})$ with the $\mathbb{C}$-bilinear extension of $g_\eta$ (see \cite{CR}, Chapter 1).

\begin{remark}
In the case $n=1$,  we have
\begin{equation}\label{eqH3}
\bH^{3}=\lbrace z=(z_1,z_2)\in \mathbb{C}^2, |z_2|^2-|z_1|^2 =1\rbrace.
\end{equation}
It is isomorphic to the group ${\rm{SL}}(2,\R)=\left\{\left(
\begin{array}{ccc}
a&b \\
c&d
\end{array}
\right), a,b,c,d \in\R,ad-bc=1 \right\}$. \\
Indeed, by writing
$z_1=x_1+iy_1$,  $z_2=x_2+iy_2$, $x_1,x_2,y_1,y_2\in\R$, we see that (\ref{eqH3}) becomes
\[
x_2^2+y_2^2-x_1^2-y_1^2=1,
\]
that is 
\[
(x_1+x_2)(x_1-x_2)-(y_1+y_2)(y_1-y_2)=1.
\]
By denoting
\[
a=x_1+x_2, d=x_1-x_2, b=y_1+y_2, c=y_1-y_2, 
\]
we obtain the isomorphism between these two spaces.

With this special group structure, the sub-Laplacian and the corresponding heat kernel on $\bH^{3}$ is studied by Bonnefont in \cite{Bo}.
In \cite{CMV}, the geometric aspects of $\bH^3$ are studied in detail. 
\end{remark}

\subsection{The cylindric coordinates and the sub-Laplacian on $\bH^{2n+1}$}
To derive the canonical sub-Laplacian that is associated to the sub-Riemannian structure on $\bH^{2n+1}$, we first introduce  
 a set of coordinates that takes into account the symmetries of the fibration $\bS^1\to\mathbb{H}^{2n+1} \to \mathbb{CH}^n$.

Let $(w_1,\cdots, w_n,\theta)$ be local coordinates for $\bH^{2n+1}$, where $(w_1,\cdots,w_n)$ are the local coordinates for $\mathbb{CH}^n$ given by $w_j=z_j/z_{n+1}$ and $\theta$ is the local fiber coordinate on $\bS^1$. Namely, $(w_1, \cdots, w_n)$ parametrizes the complex lines passing through the origin, while $\theta$ determines a point on the line that is of unit distance from the north pole\footnote{
We call north pole the point with complex coordinates $z_1=0,\dots, z_{n+1}=1$, it is therefore the point with real coordinates $(0,\dots,0,1,0)$.
}. 
More explicitly, we have
\begin{align}\label{cylinder}
(w_1,\dots,w_n,\theta)\longrightarrow \left(\frac{w_1e^{i\theta}}{\sqrt{1-\rho^2}},\dots,\frac{w_ne^{i\theta}}{\sqrt{1-\rho^2}},\frac{e^{i\theta}}{\sqrt{1-\rho^2}} \right),
\end{align}
where $\rho=\sqrt{\sum_{j=1}^{n}|w_j|^2}$, $\theta \in \R/2\pi\mathbb{Z}$, and $w=(w_1,\dots, w_n) \in \mathbb{CH}^n$. These are referred to as the cylindric coordinates throughout this paper. 
%\begin{remark}
%Another intuitive way to consider the cylindric coordinates is from the view of ``sausage coordinates'' that are introduced in Bengtsson's notes (see \cite{Bengtsson},  P31), where the anti-de Sitter space is seen as a salami whose slices are hyperbolic spaces. In our cylindric coordinates the situation is slightly different: the slices are anti-de Sitter spaces (of lower dimension) instead of hyperbolic spaces. But our fiber coordinate $\theta$ indeed coincides with the periodic time-like coordinate in \cite{Bengtsson}.  
%\end{remark}
In the cylindric coordinates, we clearly have that 
\begin{equation}\label{eqTT}
T=\frac{\partial}{\partial \theta}.
\end{equation}
By using the diffeomorphism 
\begin{align*}
(w_1,\dots,w_n,\theta,\kappa)\longrightarrow \left(\frac{\kappa w_1e^{i\theta}}{\sqrt{1-\rho^2}},\dots,\frac{\kappa w_ne^{i\theta}}{\sqrt{1-\rho^2}},\frac{\kappa e^{i\theta}}{\sqrt{1-\rho^2}} \right),
\end{align*}
and then restricting to the surface where 
\[
\kappa=1,\  \frac{\partial }{\partial \kappa}=0,
\]
we can compute that  for $1\leq k\leq n$,
\begin{eqnarray*}
\frac{\partial}{\partial z_k} &=&\sqrt{1-\rho^2}e^{-i\theta}\frac{\partial}{\partial w_k},\\
\frac{\partial}{\partial \overline{z_k}} &=&\sqrt{1-\rho^2}e^{i\theta}\frac{\partial}{\partial \overline{w_k}}, \\
\end{eqnarray*}
and
\begin{eqnarray*}
\frac{\partial}{\partial z_{n+1}} &=&-\sqrt{1-\rho^2}e^{-i\theta}\left(\sum_{j=1}^n w_j\frac{\partial}{\partial w_j}-\frac{1}{2i}\frac{\partial}{\partial \theta}\right), \\
\frac{\partial}{\partial \overline{z_{n+1}}}  &=& -\sqrt{1-\rho^2}e^{i\theta}\left(\sum_{j=1}^n\overline{w_j}\frac{\partial}{\partial \overline{w_j}}+\frac{1}{2i}\frac{\partial}{\partial \theta}\right).
\end{eqnarray*}
Moreover, from (\ref{eqeta}) we have that the contact form in cylindric coordinates is given by
\begin{equation}\label{eq-eta-cylin}
\eta=-d\theta+\frac{i}{2(1-\rho^2)}\sum_{j=1}^n(w_jd\overline{w_j}-\overline{w_j}dw_j),
\end{equation}
and the Levi form  is 
\begin{equation}\label{eqL}
\mathcal{L}_\eta=-\frac{i}{2}d\eta=\frac{1}{2(1-\rho^2)}\sum_{j=1}^ndw_j\wedge d\overline{w_j}+\frac{1}{2(1-\rho^2)^2}\sum_{j,k=1}^n\overline{w_k}w_jdw_k\wedge d\overline{w_j}.
\end{equation}
Now to construct a basis $\lbrace T_k\rbrace_{k=1}^{n}$ for $T^{1,0}(\bH^{2n+1})$, we can lift a basis $\lbrace \frac{\partial}{\partial w_k}\rbrace_{k=1}^n$ of $T^{1,0}(\mathbb{CH}^n)$ to $T^{1,0}(\bH^{2n+1})$ by using the fact  that $T_k\in \mathbf{ker}(\eta)$, $k=1,\cdots, n $. Easy calculations show that
\begin{equation}\label{eqT}
T_k=
\frac{\partial}{\partial w_k}+\frac{\overline{w_k}}{2i(1-\rho^2)}\frac{\partial}{\partial\theta}.
\end{equation}
We denote by $\lbrace \overline{T}_k\rbrace_{k=1}^{n}$ the conjugate of $\lbrace T_k\rbrace_{k=1}^{n}$,  which form a basis of $T^{0,1}(\bH^{2n+1})$. The horizontal distribution $\mathcal{H}(\bH^{2n+1})={\mathrm{Re}}\{T^{1,0}(\bH^{2n+1})\oplus T^{0,1}(\bH^{2n+1})\}$ is generated by space-like vector fields and the Reeb vector field $T$ is the time-like direction. Moreover, we can easily see that $\mathcal{L}_\eta$ is strictly pseudo-convex on $T^{1,0}(\bH^{2n+1})\oplus T^{0,1}(\bH^{2n+1})$: for any $c=(c_1,\dots, c_n)\in\mathbb{C}^n$, $c\not=0$,
\[
L_\eta(c_1T_1+\dots, c_nT_n, \bar{c}_1\overline{T}_1+\dots+\bar{c}_n\overline{T}_n)\ge\frac{\|c\|^2}{2(1-\rho^2)}>0.
\] 
Now we are ready to compute the sub-Laplacian $L$ on $\bH^{2n+1}$. Let $\Psi=\eta\wedge(d\eta)^n$ be the corresponding volume form. The sub-Laplacian $L$ is given by 
\begin{equation}\label{eq-div-grad}
Lu=\mathrm{div}(\nabla_{\mathcal{H}}u)
\end{equation} 
for any $u\in C^2(\bH^{2n+1})$, where $\nabla_{\mathcal{H}}$ is the horizontal gradient (projection of gradient $\nabla$ on $\mathcal{H}(\bH^{2n+1})$).
%Let us recall that on a %strict pseudo-convex 
%CR manifold $M$ endowed with a pseudo-Hermitian structure $\eta$, the sub-Laplacian $L$ is given by (see \cite{CR})
%\[
%Lu=\mathbf{trace}_{\mathcal{L}_\eta}\lbrace \pi_H\nabla^2u\rbrace
%\]
%for any $u\in C^2(\bH^{2n+1})$. Here the pseudo-Hermitian Hessian $\nabla^2$ of $u$ is defined by
%\[
%\left(\nabla^2 u\right)(X,Y)=X(Y(u))-(\nabla_XY)(u), 
%\] 
%or any vector fields $X$,$Y$ on $\bH^{2n+1}$.
%Now we deduce the following expression of $L$ on $\bH^{2n+1}$.

\begin{proposition}\label{prop2}
Let $L$ be the sub-Laplacian on $\bH^{2n+1}$ associated to the contact form $\eta$. Then 
\begin{equation}\label{eq-L-cylin}
L=4(1-\rho^2)\sum_{k=1}^n \frac{\partial^2}{\partial w_k \partial\overline{w_k}}- 4(1-\rho^2)\mathcal{R} \overline{\mathcal{R}}+\rho^2\ \frac{\partial^2}{\partial \theta^2}+2i(1-\rho^2)(\mathcal{R} -\overline{\mathcal{R}})\frac{\partial}{\partial\theta},
\end{equation}
where $\mathcal{R}=\sum_{k=1}^nw_k\frac{\partial}{\partial w_k}$.
\end{proposition}
\begin{proof}
To compute explicitly the sub-Laplacian, we use another expression that is more convenient for us. For any local frame $\lbrace T_k\rbrace_{k=1}^n$ of $T^{1,0}(\bH^{2n+1})$, if we denote $h_{k\overline{j}}=\mathcal{L}_{\eta}(T_k,\overline{T}_j)$, then
\begin{equation}\label{eqLu}
Lu=\sum_{k,j=1}^n \left(h^{k\overline{j}}(\nabla^2 u)(T_k,\overline{T_j})+h^{\overline{k}j}(\nabla^2u)(\overline{T_k},T_j) \right),
\end{equation}
where the matrix $[h^{\overline{k}j}]$ is the inverse of $[h_{k\overline{j}}]$, i.e., $[h^{\overline{k}j}] =[h_{k\overline{j}}]^{-1}$, $h^{\overline{k}j}h_{j\overline{l}}=\delta_{kl}$ for all $1\le j,k\le n$. It is known to be an equivalent expression of the sub-Laplacian \eqref{eq-div-grad} (see \cite{CR}, P117 for detailed proof). From \eqref{eqL}, we can compute that
\begin{equation}\label{eqh}
h^{\overline{k}j}=2(1-\rho^2)(\delta_{kj}-\overline{w_k}w_j),\quad k,j=1,\cdots,n.
\end{equation}
Plugging  (\ref{eqT}) and (\ref{eqh}) into (\ref{eqLu}), we obtain
\[
L=4(1-\rho^2)\sum_{k=1}^n \frac{\partial^2}{\partial w_k \partial\overline{w_k}}- 4(1-\rho^2)\mathcal{R} \overline{\mathcal{R}}+\rho^2\ \frac{\partial^2}{\partial \theta^2}+2i(1-\rho^2)(\mathcal{R} -\overline{\mathcal{R}})\frac{\partial}{\partial\theta}.
\]
\end{proof}
From (\ref{eqL}) we know that the volume form $\Psi$ in cylindric coordinates is given by 
\[
\Psi=\eta\wedge(d\eta)^n=\frac{i^{n^2}n!}{(1-\rho^2)^{2n}}\mathbf{det}(\alpha_{k\overline{j}})d\theta\wedge dw_1\wedge\cdots\wedge dw_n\wedge d\overline{w_1}\wedge\cdots\wedge d\overline{w_n}
\]
where $\alpha_{k\overline{k}}=1, \alpha_{k\overline{j}}=\overline{w_k}w_j$,  for $1\leq k, j\leq n$, $k\not=j$. Since $\bH^{2n+1}$ is complete, we know that $L$ is essentially self-adjoint on $C^{\infty}_0(\bH^{2n+1})$  (see  \cite{S-83} and \cite{S}).\\

Next let us consider the Laplacian on $\bH^{2n+1}$ with respect to the restricted Lorentzian metric \eqref{eq-lorent} and denote it by $\square$.  If we denote the Laplacian on the flat Lorentzian space $\R^{2n,2}$ by $\square_\mathbb{C}$, then $\square$ can be obtained by restricting $\square_\mathbb{C}$ to $\bH^{2n+1}$. In complex coordinates $\square_\mathbb{C}$  is given by
\[
\square_{\mathbb{C}}=4\left(\sum_{k=1}^n\frac{\partial^2}{\partial z_k\partial \overline{z_k}}-\frac{\partial^2}{\partial z_{n+1}\partial\overline{z_{n+1}}}\right).
\]
Let $f:\bH^{2n+1}\to\mathbb{C}$ be a smooth function. It may be extended to $\tilde{f}$ on $\mathbb{C}^{n+1}\setminus \lbrace 0\rbrace$ by
\[
\tilde{f}(z)=f\left( \frac{z}{\|z\|}\right),
\]
and  we have 
\[
\square f(z)=\square_{\mathbb{C}}\tilde{f}(z),\quad z\in \bH^{2n+1}.
\] 
Simple but tedious calculations show that
\[
\square=4\left(\sum_{k=1}^n\frac{\partial^2}{\partial z_k\partial \overline{z_k}}-\frac{\partial^2}{\partial z_{n+1}\partial\overline{z_{n+1}}}\right)-(\mathcal{S}+\overline{\mathcal{S}})^2+2n(\mathcal{S}+\overline{\mathcal{S}}).
\]
where $\mathcal{S}=\sum_{k=1}^nz_k\frac{\partial}{\partial z_k}$. By plugging in \eqref{cylinder} we then obtain  
\[
 \square=4(1-\rho^2)\sum_{k=1}^n \frac{\partial^2}{\partial w_k \partial\overline{w_k}}- 4(1-\rho^2)\mathcal{R} \overline{\mathcal{R}}-(1-\rho^2)\ \frac{\partial^2}{\partial \theta^2}+2i(1-\rho^2)(\mathcal{R} -\overline{\mathcal{R}})\frac{\partial}{\partial\theta}.
\]
Since $T=\frac{\partial}{\partial\theta}$, by comparing the above expression to \eqref{eq-L-cylin} we immediately obtain that 
\begin{equation}\label{LT}
L=\square +T^2.
\end{equation}

\begin{remark}\label{Rk-Sasakian}
We can observe a fact which will be important for us:  $L$ and $T$ commute. Namely for any smooth functions we have
\[
TL=LT.
\]
This is equivalent to the transverse symmetry of $\bH^{2n+1}$ (Sasakian structure of $\bH^{2n+1}$). In other words, this is equivalent to the vanishing pseudo-Hermitian torsion of $\bH^{2n+1}$.
\end{remark}

In the study of the heat kernel, due to the radial symmetry of the fibration $\bH^{2n+1} \to \mathbb{CH}^n$, it will be enough for us  to compute the radial part of $L$ with variables $(\rho,\theta)$.
From \eqref{eq-L-cylin} we consider the following operator:
\begin{equation}\label{eq-L-rho}
 \tilde{L}=\left(1-\rho^2\right)^2\frac{\partial^2}{\partial \rho^2}+(1-\rho^2)\left(\frac{(2n-1)}{\rho}-\rho\right)\frac{\partial}{\partial \rho}+\rho^2\frac{\partial^2 }{\partial \theta^2}.
\end{equation}
$\tilde{L}$ is defined on the space $\mathcal{D}$, which consists of the smooth functions $ f :  \mathbb{R}_{\ge 0} \times \R/2\pi\mathbb{Z}  \to \mathbb{R}$  satisfying $\frac{\partial f}{\partial  \rho} =0$  when $\rho =0$.  In the next proposition we prove that $\tilde{L}$ is indeed the radial part of $L$.

\begin{proposition} Let us denote by $\psi$ the map from $\bH^{2n+1} $ to $  \mathbb{R}_{\ge 0} \times \R/2\pi\mathbb{Z} $ satisfying 
\[
\psi \left(\frac{w_1e^{i\theta}}{\sqrt{1-\rho^2}},\dots,\frac{w_ne^{i\theta}}{\sqrt{1-\rho^2}},\frac{e^{i\theta}}{\sqrt{1-\rho^2}} \right)=\left(\rho, \theta \right).
\]
 For every  $f \in \mathcal{D}$,  we have
\[
L(f \circ \psi)=(\tilde{L} f) \circ \psi.
\]
\end{proposition}

\begin{proof}
Notice that by the radial symmetry  we have
\[
\left(\sum_{k=1}^n\frac{\partial^2}{\partial w_k\partial\overline{w_k}}\right)(f\circ\psi)=\left(\left(\frac{1}{4}\frac{\partial^2}{\partial\rho^2}+\frac{2n-1}{4\rho}\frac{\partial}{\partial \rho}\right)f\right)\circ\psi
\]
and
\[
\mathcal{R}(f\circ\psi)=\overline{\mathcal{R}}(f\circ\psi)=\left(\left(\frac{1}{2}\rho\frac{\partial}{\partial \rho}\right)f\right)\circ\psi.
\]
Together with Proposition \ref{prop2}, we have the conclusion.
\end{proof}

The volume measure in cylindric coordinates is correspondingly given by $\frac{\rho^{2n-1}}{(1-\rho^2)^{n+1 }} d\rho d\theta$. $\tilde{L}$ is then symmetric and essentially self-adjoint on $\mathcal{D}$ with respect to the above volume measure. 
Moreover, instead of $\rho$, it will be expedient to introduce a more geometrically meaningful variable $r$ such that
\[
\rho=\tanh r.
\] 
Indeed $r$ is the Riemannian distance on $\mathbb{CH}^n$ from its north pole. From \eqref{eq-L-rho} we then have
\begin{equation}\label{eq3}
\tilde{L}=\frac{\partial^2}{\partial r^2}+((2n-1)\coth r+\tanh r)\frac{\partial}{\partial r}+\tanh^2r\frac{\partial^2}{\partial \theta^2}.
\end{equation}
This expression shall be the most convenient one for us and be used throughout the paper. If we denote the corresponding volume measure in the coordinates $(r,\theta)$ by $\mu_r$, then
\[
d\mu_r=\frac{2\pi^n}{\Gamma(n)}(\sinh r)^{2n-1}\cosh r drd\theta.
\]

\begin{example}
When $n=1$, \eqref{eq3} is 
\begin{eqnarray*}
\tilde{L}&=&\frac{\partial^2}{\partial r^2}+2\coth 2r\frac{\partial}{\partial r}+\tanh^2r\frac{\partial^2}{\partial\theta^2}.
\end{eqnarray*}
This gives the radial sub-Laplacian of ${\rm{SL}}(2,\R)$, which coincides with the result in \cite{Bo}.
\end{example}

We can easily compute the radial part $\tilde{\square}$ of the Laplacian on $\bH^{2n+1}$: 
\begin{equation}\label{square}
\tilde{\square}=\frac{\partial^2}{\partial r^2}+((2n-1)\coth r+\tanh r)\frac{\partial}{\partial r}-\frac{1}{\cosh^2r}\frac{\partial^2}{\partial\theta^2}.
\end{equation}
Recall in cylindric coordinates   $z_{n+1}=\cosh r e^{i\theta}$.  Clearly the Riemannian distance $\delta$ on $\bH^{2n+1}$ from its north pole satisfies
\[
\cosh \delta =-\cosh  r \cos \theta.
\]
Indeed if we plug the above equation into \eqref{square}, we may recover the expression of  $\tilde{\square}$ in spherical coordinates:
 \[
 \tilde{\square}=\frac{\partial^2}{\partial \delta^2}+2n\coth\delta\frac{\partial}{\partial\delta}
 \]
 with symmetric measure $\frac{2\pi^n}{\Gamma(n)}(\sinh\delta)^{2n}d\delta$. 

\subsection{Laplacian on the hyperbolic space and the analytic continuation}

Another important fact for us is that an anti-de Sitter space can be seen as the analytic continuation  of a hyperbolic space. In this section, we present a sketch of the results (more background descriptions can be found in \cite{Bengtsson}, P47).

A hyperbolic space $\mathbf{H}^{2n+1}$ is defined as the upper sheet of a two-sheeted hyperboloid 
\[
x_1^2+\cdots+x_{2n+1}^2-x_{2n+2}^2=-1
\]
embedded in a flat $2n+2$-dimensional space $\R^{2n+1,1}$ with the metric of  Lorentzian signature $(2n+1,1)$:
\[
ds^2=dx_1^2+\cdots+dx_{2n+1}^2-dx_{2n+2}^2.
\] 
$\mathbf{H}^{2n+1}$ appears as an analytic continuation of the anti-de Sitter space $\bH^{2n+1}$. We observe it as follows: let $(x_1,y_1,\dots, x_{n+1}, y_{n+1})$ be the real coordinates of $\bH^{2n+1}$ in the ambient space $\R^{2n+1,1}$, then from \eqref{cylinder} we have,
\[
x_k=\frac{u_k\cos(\theta)-v_k\sin(\theta)}{\sqrt{1-\rho^2}}, \ y_k=\frac{v_k\cos(\theta)+u_k\sin(\theta)}{\sqrt{1-\rho^2}},\  k=1,\dots,n
\]
and 
\[
x_{n+1}=\frac{\cos(\theta)}{\sqrt{1-\rho^2}}, \ y_{n+1}=\frac{\sin(\theta)}{\sqrt{1-\rho^2}},
\]
where $u_k=\mathbf{Re}(w_k)$ and $v_k=\mathbf{Im}(w_k)$, $k=1,\dots, n$.
If we consider the analytic continuation of the metric of Lorentzian signature to the one of Euclidean metric, by letting the time-like parameter
\begin{equation}\label{eq-theta-E}
\theta\to -iy,
\end{equation}
then the resulting space is indeed a $2n+1$-dimensional hyperbolic space $\mathbf{H}^{2n+1}$:
\[
X_1^2+Y_1^2+\cdots+Y_n^2+X_{n+1}^2-Y_{n+1}^2=-1,
\]
where 
\[
X_k=\frac{u_k\cosh(y)-v_k\sinh(y)}{\sqrt{1-\rho^2}}, \ Y_k=\frac{v_k\cosh(y)+u_k\sinh(y)}{\sqrt{1-\rho^2}},\  k=1,\dots,n
\]
and 
\[
X_{n+1}=\frac{\cosh(y)}{\sqrt{1-\rho^2}}, \ Y_{n+1}=\frac{\sinh(y)}{\sqrt{1-\rho^2}}.
\]
The ambient space $\R^{2n+1,1}$ of $\mathbf{H}^{2n+1}$  is in fact the analytic continuation of the ambient space $\R^{2n,2}$ of $\bH^{2n+1}$  (see \eqref{eq-lorent}). The metric on $\R^{2n+1,1}$ restricted to the tangent bundle of $\mathbf{H}^{2n+1}$ is the Riemannian metric with Euclidean signature, namely $(2n+1, 0)$. 
 
From this point of view, an anti-de Sitter space appears as a hyperbolic model embedded in the flat Minkowski space $\R^{2n+1,1}$. Its topology is $\R^{2n}\times \bS^1$. $\bH^{2n+1}$ can also be seen as the counterpart of a de Sitter space in the same ambient space $\R^{2n+1,1}$, where the latter one appears as an elliptic model. Indeed a $2n+1$-dimensional de Sitter space is defined as a one sheeted hyperboloid 
\[
X_1^2+\cdots+X_{2n}^2-X_{2n+1}^2=1
\]
embedded in $\R^{2n+1,1}$, whose topology is given by $S^{2n}\times \R$. The analytic continuation of a $2n+1$-dimensional de Sitter space is the Riemannian sphere $S^{2n+1}$ (for more details, see \cite{Bengtsson}).

Now let us consider the Laplacian on a $2n+1$-dimensional hyperbolic space $\mathbf{H}^{2n+1}$ and denote it by $\Delta$. From the argument of analytic continuation, we immediately obtain the radial part of $\Delta$ in cylindric coordinates
\[
\tilde{\Delta}=\frac{\partial^2}{\partial r^2}+((2n-1)\coth r+\tanh r)\frac{\partial}{\partial r}+\frac{1}{\cosh^2r}\frac{\partial^2}{\partial y^2},
\]
where $y$ is the time-like parameter in Euclidean space and satisfies \eqref{eq-theta-E}.  $\tilde{\Delta}$ is an elliptic operator defined $(0,+\infty)\times\R$ and is essentially self-adjoint with respect to the measure $\frac{2\pi^n}{\Gamma(n)}(\sinh r)^{2n-1}\cosh rdrdy$. 
Also if we denote by $\delta'$ the Riemannian distance on $\mathbf{H}^{2n+1}$ from its north pole, then
\[
\cosh{\delta'}=\cosh r\cosh y.
\]

We finish this section by noticing  that $\bH^{2n+1}$ is not simply connected. However, we can always unwrap it by studying its universal covering $\widetilde{\bH^{2n+1}}$.
In this case, the fiber of the $\widetilde{\bH^{2n+1}}$ is $\R$ and the Hopf fibration is given by
\[
\R\longrightarrow\widetilde{\bH^{2n+1}}\longrightarrow \mathbb{CH}^n.
\]
The cylindric coordinates on $\widetilde{\bH^{2n+1}}$ are $(w_1,\cdots,w_n,\theta)\in\mathbb{C}^n\times\R$, and the projection from $\widetilde{\bH^{2n+1}}$ to $\bH^{2n+1}$ can be obtained by projecting  the fiber coordinate $\theta$ from $\R$ to $\R/2\pi\mathbb{Z}$. 
Moreover, the horizontal distribution and the Reeb vector field on $\bH^{2n+1}$ can be lifted to $\widetilde{\bH^{2n+1}}$. The expressions remain the same as in (\ref{eqT}) and (\ref{eqTT}), and are defined for all $(w_1,\cdots,w_n,\theta)\in\mathbb{C}^n\times\R$.
The sub-Laplacian as well as its radial part defined on $\widetilde{\bH^{2n+1}}$ obviously share the same expressions as the ones on $\bH^{2n+1}$. In the rest of the paper, we shall use the same notation for both cases.

\section{The subelliptic heat kernel on $\bH^{2n+1}$ and $\widetilde{\bH^{2n+1}}$}

\subsection{Integral representation of the heat kernel}

In this section we derive a integral representation of the subelliptic heat kernel on $\bH^{2n+1}$ (and  $\widetilde{\bH^{2n+1}}$) in terms of the Riemannian heat kernel on $\mathbf{H}^{2n+1}$ associated to $\tilde{\Delta}$. The Riemannian heat kernel on $\mathbf{H}^{2n+1}$  issued from its north pole is given explicitly by Gruet (see \cite{Gr}):
\begin{equation}\label{eq6}
q_t(\cosh\delta')=\frac{\Gamma(n+1)e^{-n^2t}}{(2\pi)^{n+1}\sqrt{\pi t}}\int_0^{+\infty}\frac{e^{\frac{\pi^2-u^2}{4t}}\sinh u\sin{\frac{\pi u}{2t}}}{\left(\cosh u+\cosh\delta'\right)^{n+1}}du,
\end{equation}
where $\delta'$ is the Riemannian distance on $\mathbf{H}^{2n+1}$ from the north pole. 
Another useful formula for the heat kernel $q_t$ which shall be used later in obtaining the small time asymptotics is (see \cite{T2}, P125):
\begin{equation}\label{heat_kernel_odd}
q_t (\cosh \delta')= \frac{e^{-n^2t}}{\sqrt{4\pi t}} \left( -\frac{1}{2\pi \sinh \delta'} \frac{\partial}{\partial \delta'} \right)^n \left(e^{-\frac{\delta'^2}{4t}}\right).
\end{equation}
Since $q_t$ satisfies the heat equation:
\begin{equation}\label{eqheat}
\frac{\partial }{\partial t}q_t(\cosh r\cosh y)=\tilde{\Delta}q_t(\cosh r\cosh y),
\end{equation}
by plugging in \eqref{eq-theta-E} we easily obtain that 
\[
\frac{\partial }{\partial t}q_t(\cosh r\cos \theta)=\tilde{\square}q_t(\cosh r\cos \theta),
\]
where $\tilde{\square}$ is the radial part of the Laplacian on $\bH^{2n+1}$ (and $\widetilde{\bH^{2n+1}}$) given by \eqref{square}. \\

Now we are ready to deduce the integral representation of the subelliptic heat kernel on  $\widetilde{\bH^{2n+1}}$. The key idea is to take advantage of the intertwining between $L$ and $T$ and the relation  described in  \eqref{LT}.  We have a heuristic observation:
\begin{align}\label{commutation}
e^{tL}=e^{t\frac{\partial^2}{\partial\theta^2}}e^{t\square}.
\end{align}
It  suggests that the sub-Riemannian heat kernel $p_t^{\widetilde{\bH^{2n+1}}}$ on  $\widetilde{\bH^{2n+1}}$ can be expressed as the Riemannian heat kernel on   $\widetilde{\bH^{2n+1}}$ convolved with a Gaussian kernel, from which we can then express $p_t^{\widetilde{\bH^{2n+1}}}$ in terms of  Riemannian heat kernel on $\mathbf{H}^{2n+1}$ by using the  analytic continuation argument. 
\begin{prop}\label{prop1}
For $t>0$, $r\in[0,+\infty)$, $ \theta\in(-\infty,+\infty)$, the subelliptic heat kernel on  $\widetilde{\bH^{2n+1}}$ is given by
\begin{equation}\label{eq8}
p_t^{\widetilde{\bH^{2n+1}}}(r, \theta)=\frac{1}{\sqrt{4\pi t}}\int_{-\infty}^{+\infty}e^{\frac{(y-i \theta)^2}{4t} }q_t(\cosh r\cosh y)dy.
\end{equation}
More precisely,
\begin{equation}\label{eq-pt-Gr}
p_t^{\widetilde{\bH^{2n+1}}}(r, \theta)=\frac{\Gamma(n+1)e^{-n^2t+\frac{\pi^2}{4t}}}{(2\pi)^{n+2} t}\int_{-\infty}^{+\infty}\int_0^{+\infty}\frac{e^{\frac{(y-i\theta)^2-u^2}{4t}}\sinh u\sin\left(\frac{\pi u}{2t}\right)}{\left(\cosh u+\cosh r\cosh y\right)^{n+1}}du
dy.
\end{equation}
\end{prop}
\begin{proof}
Since \eqref{eq-pt-Gr} can be simply obtained from \eqref{eq6} and \eqref{eq8}, we just need to prove that \eqref{eq8} is the desired subelliptic heat kernel. Before proceeding to the proof, let us first check the existence of the integral in  \eqref{eq8}. Noting that $e^{\frac{(y-i\theta)^2}{4t}}$ blows up as $y\to\pm\infty$, we need to show that $q_t$ decays fast enough to compensate $e^{\frac{(y-i\theta)^2}{4t}}$. From \eqref{heat_kernel_odd} we can easily see that $q_t(\cosh \delta)$ decays at least at the scale of $e^{-\frac{\delta^2}{4t}}\left(\frac{\delta}{\sinh\delta}\right)$; namely there exists a $C>0$ such that $q_t(\cosh \delta)\le Ce^{-\frac{\delta^2}{4t}}\left(\frac{\delta}{\sinh\delta}\right)$. 
Since $e^{-\frac{\delta^2}{4t}}\left(\frac{\delta}{\sinh\delta}\right)$ decreases on $(0,+\infty)$ and $\cosh^{-1}(\cosh r\cosh y)\ge y$, we have
\begin{equation}\label{eq-integrability}
q_t(\cosh r\cosh y)\le Ce^{-\frac{y^2}{4t}}\left(\frac{y}{\sinh y}\right).
\end{equation}
Hence we can conclude the integrability of \eqref{eq8} by
\[
\int_{-\infty}^{+\infty}\bigg|e^{\frac{(y-i \theta)^2}{4t} }q_t(\cosh r\cosh y)\bigg|dy\le C\int_{-\infty}^{+\infty} \left(\frac{y}{\sinh y}\right)dy<+\infty.
\]
Also the decay of $q_t$ allows us to integrate by part, differentiate under the integral sign, and interchange integrals  later.

Next we want to show \eqref{eq8}. Denote
\[
h_t(r, \theta)=\frac{1}{\sqrt{4\pi t}}\int_{-\infty}^{+\infty}e^{\frac{(y-i \theta)^2}{4t} }q_t(\cosh r\cosh y)dy.
\]
Our first goal is to prove that $h_t$ satisfies the heat equation $\frac{\partial}{\partial t}h_t=\tilde{L}h_t$, where $\tilde{L}$ is given in \eqref{eq3}. 
We denote $L_0=\frac{\partial^2}{\partial r^2}+((2n-1)\coth r+\tanh r)\frac{\partial}{\partial r}$. Then by comparing with (\ref{eqheat}), we have
\[
\frac{\partial}{\partial t}\left(q_t(\cosh r\cosh y) \right)=\left(L_0+\frac{1}{\cosh^2 r}\frac{\partial^2}{\partial y^2}\right)(q_t(\cosh r\cosh y)).
\]
Since
\begin{equation}\label{dt}
\frac{\partial}{\partial t}\left(\frac{e^{\frac{(y-i\theta)^2}{4t} }}{\sqrt{4\pi t}}\right)=\frac{\partial^2}{\partial \theta^2}\left(\frac{e^{\frac{(y-i\theta)^2}{4t}} }{\sqrt{4\pi t}}\right)
=-\frac{\partial^2}{\partial y^2}\left(\frac{e^{\frac{(y-i\theta)^2}{4t}} }{\sqrt{4\pi t}}\right),
\end{equation}
by integrating by parts twice with respect to $y$, we have 
\begin{equation}\label{eq-qt-L0}
\frac{\partial}{\partial t}\left(h_t(r, \theta) \right)=\frac{1}{\sqrt{4\pi t}}\int_{-\infty}^{+\infty}e^{\frac{(y-i \theta)^2}{4t} }\left(L_0-\tanh^2 r\frac{\partial^2}{\partial y^2}\right)q_t(\cosh r\cosh y)dy.
\end{equation}
On the other hand, since $\tilde{L}=L_0+\tanh^2r\frac{\partial^2}{\partial \theta^2}$, we have
\[
\tilde{L}(h_t(r, \theta))=\int_{-\infty}^{+\infty}\left(\left(\tanh^2r\frac{\partial^2}{\partial \theta^2}\left(\frac{e^{\frac{(y-i\theta)^2}{4t} }}{\sqrt{4\pi t}}\right)\right)q_t(\cosh r\cosh y)+\left(\frac{e^{\frac{(y-i\theta)^2}{4t} }}{\sqrt{4\pi t}}\right)L_0\left(q_t(\cosh r\cosh y)\right)\right)dy.
\]
By (\ref{dt}) and twice integrating by parts, we have
\[
\tilde{L}(h_t(r, \theta))=\frac{1}{\sqrt{4\pi t}}\int_{-\infty}^{+\infty}e^{\frac{(y-i \theta)^2}{4t} }\left(L_0-\tanh^2 r\frac{\partial^2}{\partial y^2}\right)q_t(\cosh r\cosh y)dy.
\]
Comparing with \eqref{eq-qt-L0}, we then obtain
\[
\tilde{L}(h_t(r, \theta))=\frac{\partial}{\partial t}\left(h_t(r, \theta) \right).
\]

It remains to check the initial condition. Since $h_t(r,\theta)$ is the heat kernel initiated from north pole $(0,0)$, we just need to show that $\lim_{t\to0}h_t\ast f(0,0)=f(0,0)$, where $h_t\ast f$ is given by $(h_t\ast f)(r', \theta')= \int_{r>0}\int_{\theta=-\infty}^{+\infty}h_t(r, \theta)f(r'-r, \theta'-\theta)d\mu_r $ for all $r'>0$, $\theta'\in\R$.
It suffices to check the initial condition for functions of the form $f(r, \theta)=e^{i\lambda \theta}g(r)$ where $\lambda\in\mathbb{R}$ and $g:(0,+\infty)\to \R$ is smooth.  We have
\begin{eqnarray*}
& & (h_t\ast f)(0,0)= \int_{r>0}\int_{\theta=-\infty}^{+\infty}h_t(r, \theta)f(r, \theta)d\mu_r \\
&=& \frac{2\pi^n}{\Gamma(n)}\int_{r>0}\int_{\theta=-\infty}^{+\infty}\int_{y>0}\left( \frac{e^{-\frac{(\theta+iy)^2}{4t}}+e^{-\frac{(\theta-iy)^2}{4t}}}{\sqrt{4\pi t}} \right)q_t(\cosh r\cosh y) e^{i\lambda\theta}g(r)(\sinh r)^{2n-1}\cosh rdyd\theta dr.
\end{eqnarray*}
By changing the contour of the integral, we get
\[
\int_{\theta=-\infty}^{+\infty}\left(\frac{e^{-\frac{(\theta-iy)^2}{4t}}}{\sqrt{4\pi t}} \right)e^{i\lambda\theta}d\theta=e^{-\lambda y}\int_{\theta=-\infty}^{+\infty}\left(\frac{e^{-\frac{\theta^2}{4t}}}{\sqrt{4\pi t}} \right)e^{i\lambda\theta}d\theta=e^{-\lambda y-\lambda^2 t}
\]
and
\[
\int_{\theta=-\infty}^{+\infty}\left(\frac{e^{-\frac{(\theta+iy)^2}{4t}}}{\sqrt{4\pi t}} \right)e^{i\lambda\theta}d\theta=e^{\lambda y}\int_{\theta=-\infty}^{+\infty}\left(\frac{e^{-\frac{\theta^2}{4t}}}{\sqrt{4\pi t}} \right)e^{i\lambda\theta}d\theta=e^{\lambda y-\lambda^2 t}.
\]
Hence
\begin{eqnarray*}
 (h_t\ast f)(0,0) &= & \frac{4\pi^n }{\Gamma(n)}e^{-\lambda^2 t}\int_{r>0}\int_{y>0}q_t(\cosh r\cosh y)g(r)\cosh(\lambda y)(\sinh r)^{2n-1}\cosh rdy dr \\
&=& e^{-\lambda^2 t}\int_{r>0}\int_{y=-\infty}^{+\infty}q_t(\cosh r\cosh y)l(r,y)d\mu_r \\
&=& e^{-\lambda^2 t}(e^{t\tilde{\Delta}}l)(0,0),
\end{eqnarray*}
where $l(r,y)=g(r)\cosh (\lambda y)$. Therefore we have that
\[
\lim_{t\to0}(h_t\ast f)(0,0)=l(0,0)=g(0)=f(0,0).
\]
Thus we can conclude that $h_t(r, \theta)$ is the desired subelliptic heat kernel.
\end{proof}

From \eqref{eq8} we can obtain the subelliptic heat kernel on $\bH^{2n+1}$ immediately. 
\begin{prop}
For $t>0$, $r\in[0,+\infty)$, $ \theta\in[-\pi,\pi]$, the subelliptic heat kernel on  $\bH^{2n+1}$ is given by
\begin{equation}\label{eqpt}
p_t^{\bH^{2n+1}}(r, \theta)=\frac{\Gamma(n+1)e^{-n^2t+\frac{\pi^2}{4t}}}{(2\pi)^{n+2} t}\sum_{k\in\mathbb{Z}}\int_{-\infty}^{+\infty}\int_0^{+\infty}\frac{e^{\frac{(y-i\theta-2k\pi i)^2-u^2}{4t}}\sinh u\sin\left(\frac{\pi u}{2t}\right)}{\left(\cosh u+\cosh r\cosh y\right)^{n+1}}du
dy.
\end{equation}
\end{prop}
\begin{proof}
Clearly $p_t^{\bH^{2n+1}}$ satisfies the heat equation $\frac{\partial }{\partial t}p_t^{\bH^{2n+1}}=\tilde{L}p_t^{\bH^{2n+1}}$. It satisfies the initial condition because 
\[
\sum_{k\in\mathbb{Z}}\int_{-\pi}^{\pi}\left(\frac{e^{-\frac{(\theta+2k\pi+iy)^2}{4t}}}{\sqrt{4\pi t}}\right)e^{i\lambda\theta}d\theta=\int_{-\infty}^{+\infty}\left(\frac{e^{-\frac{(\theta+iy)^2}{4t}}}{\sqrt{4\pi t}}\right)e^{i\lambda\theta}d\theta.
\]
\end{proof}

\subsection{Asymptotics of the subelliptic heat kernel in small time}

An important application of the integral representation of the subelliptic heat kernel is that it provides us a way to estimate the asymptotic behaviors of the subelliptic kernel as $t$ tends to $0$ (also see \cite{BB}, \cite{Bo}, \cite{BW}).  
We mainly study the small time asymptotics of the subelliptic heat kernel on the universal covering $\widetilde{\bH^{2n+1}}$, since the estimates on $\bH^{2n+1}$ will be exactly the same. This is  because  when $t$ tends to $0$, the dominating term in \eqref{eqpt} is the term of $k=0$. Therefore we shall use the notation $p_t$ for the subelliptic heat kernels of  both $\bH^{2n+1}$ and $\widetilde{\bH^{2n+1}}$ in the rest of this paper.

The strategy is to take advantage of the small time asymptotic of $q_t$ that can be easily deduced from \eqref{heat_kernel_odd}:
\begin{align}\label{eq9}
q_t(\cosh\delta')=\frac{1}{(4\pi t)^{n+\frac{1}{2}}}\left(\frac{\delta'}{\sinh\delta'}\right)^ne^{-\frac{\delta'^2}{4t}}\left(1+\left(n^2-\frac{n(n-1)(\sinh\delta'-\delta'\cosh\delta')}{\delta'^2\sinh\delta'}\right)t+O(t^2)\right).
\end{align}
The small time asymptotic of $p_t$ can then be obtained by plugging the above equation into \eqref{eq8}. However, unlike Riemannian manifolds, we shall obtain three different behaviors of the subelliptic heat kernel on $\bH^{2n+1}$: on the diagonal, on the vertical cut-locus of the north pole, and outside of the cut-locus. 

We first look at the diagonal case, when $(r,\theta)=(0,0)$. 
\begin{prop}
As $t\to 0$, 
\[
p_t(0,0)=\frac{1}{(4\pi t)^{n+1}}(A_n+B_nt+O(t^2)),
\]
where $A_n=\int_{-\infty}^{\infty}\frac{y^n}{(\sinh y)^n}dy$ and $B_n=\int_{-\infty}^{\infty}\frac{y^n}{(\sinh y)^n}\left(n^2-\frac{n(n-1)(\sinh y-y\cosh y)}{y^2\sinh y}\right)dy$.
\end{prop}
\begin{proof}
We know  that
\begin{eqnarray*}
p_t(0,0)= \frac{1}{\sqrt{4\pi t}}\int_{-\infty}^{\infty}e^{\frac{y^2}{4t}}q_t(\cosh y)dy.
\end{eqnarray*}
By plugging in (\ref{eq9}), we have the desired result.
\end{proof}

Now let us we consider the subelliptic heat kernel on the vertical cut-locus of $\bH^{2n+1}$ (and $\widetilde{\bH^{2n+1}}$), namely $r=0$ and $\theta\not=0$.  
\begin{prop}
For $ \theta\in\R$, $t\rightarrow 0$,
\[
p_t(0, \theta)=\frac{ \theta^{n-1}}{2^{3n}t^{2n}(n-1)!}e^{-\frac{2\pi \theta+ \theta^2}{4t}}(1+O(t))
\]
\end{prop}
\begin{proof}
Let $ \theta\in\R$. Then
\[
p_t(0, \theta)=\frac{1}{\sqrt{4\pi t}}\int_{-\infty}^{\infty}e^{\frac{(y-i\theta)^2}{4t}}q_t(\cosh y)dy.
\]
From  (\ref{eq9}) we have
\[
q_t(\cosh y)\sim _{t\rightarrow 0}\frac{1}{(4\pi t)^{n+\frac{1}{2}}}\left(\frac{y}{\sinh y}\right)^ne^{-\frac{y^2}{4t}},
\]
hence 
\[
p_t(0, \theta) \sim_{t\rightarrow 0}  \frac{e^{-\frac{ \theta^2}{4t} }}{(4\pi t)^{n+1}}\int_{-\infty}^{\infty}\frac{y^n}{(\sinh y)^n}e^{-\frac{iy \theta}{2t}}dy  .
\]
By the residue theorem, we obtain
\begin{eqnarray*}
\int_{-\infty}^{\infty}\frac{y^n}{(\sinh y)^n}e^{-\frac{iy \theta}{2t}}dy 
&=& -2\pi i\sum_{k\in\mathbb{Z^+}} \mathbf{Res}\left( \frac{e^{-\frac{iy \theta}{2t}}y^n}{(\sinh y)^n},-k\pi i\right)\\
&=& -2\pi i\sum_{k\in\mathbb{Z^+}} \frac{1}{(n-1)!}\frac{\partial^{n-1}}{\partial y^{n-1}}\left[\frac{e^{-\frac{iy\theta}{2t}}2^ny^n(y+k\pi i)^n}{(e^y-e^{-y})^n} \right]_{y=-k\pi i}.
\end{eqnarray*}
Let $W(y)=\frac{(y+k\pi i)^n}{(e^y-e^{-y})^n}$. It is analytic around $-k\pi i$ and satisfies
\[
W(-k\pi i)=\frac{1}{(-1)^{kn}2^n}.
\]
Hence the residue $\mathbf{Res}\left( \frac{e^{-\frac{iy \theta}{2t}}y^n}{(\sinh y)^n},-k\pi i\right)$ is given by
\[
\frac{1}{(n-1)!}\frac{\partial^{n-1}}{\partial y^{n-1}}\left[e^{-\frac{iy\theta}{2t}}2^ny^nW(y) \right]_{y=-k\pi i}.
\]
This is a product of $e^{-\frac{iy\theta}{2t}}$ and a polynomial of degree $n-1$ in $1/t$. We are only interested in the leading term, which plays the dominating role when $t\rightarrow 0$. Thus we have the following estimate:
\[
\frac{-2\pi i}{(n-1)!}\frac{\partial^{n-1}}{\partial y^{n-1}}\left[e^{-\frac{iy\theta}{2t}}2^ny^nW(y) \right]_{y=-k\pi i}
\sim_{t\rightarrow 0}
\frac{(-1)^{kn+n}\pi^{n+1}\theta^{n-1}}{(n-1)!2^{n-2}t^{n-1}}e^{-\frac{k\pi \theta}{2t}}.
\]
We can conclude that
\[ 
p_t(0, \theta)\sim_{t\rightarrow 0} \frac{e^{-\frac{ \theta^2}{4t} }}{(4\pi t)^{n+1}}\sum_{k\in\mathbb{Z}^+}\frac{(-1)^{kn+n}\pi^{n+1} \theta^{n-1}}{(n-1)!2^{n-2}t^{n-1}}e^{-\frac{k\pi \theta}{2t}};
\]
that is
\[ 
p_t(0, \theta)=\frac{ \theta^{n-1}}{2^{3n}t^{2n}(n-1)!}e^{-\frac{2\pi \theta+\theta^2}{4t}}(1+O(t)).
\]
\end{proof}

Next we consider the points that do not lie on the vertical cut-locus, i.e., $r\not=0$. First let us restrict to the horizontal plane $\{(r,0)$, $r>0\}$.
\begin{prop}
For $r\in(0,+\infty)$,  we have that
\[
p_t(r,0)= \frac{e^{-\frac{r^2}{4t}}}{(4\pi t)^{n+\frac{1}{2}}}\left(\frac{r}{\sinh r}\right)^n\sqrt{\frac{1}{r\coth r-1}}(1+O(t)).
\]
\end{prop}
\begin{proof}
By proposition \ref{prop1},
\[
p_t(r,0)=\frac{1}{\sqrt{4\pi t}}\int_{-\infty}^{\infty}e^{\frac{y^2}{4t}}q_t(\cosh r\cosh y)dy.
\]
Together with (\ref{eq9}), we have
\[
p_t(r,0)\sim_{t\rightarrow 0}\frac{1}{(4\pi t)^{n+1}}\int_{-\infty}^{+\infty}e^{-\frac{(\cosh^{-1} (\cosh r\cosh y))^2-y^2}{4t}}
\left(\frac{\cosh^{-1} (\cosh r\cosh y)}{\sqrt{\cosh^2r\cosh^2y-1}} \right)^ndy.
\]
We can analyze it by the Laplace method. Notice that on $\R$, the function
\[
f(y)=(\cosh^{-1} (\cosh r\cosh y))^2-y^2
\]
has a unique minimum at $y=0$ where $f(0)=r^2$ and $f^{''}(0)=2(r\coth r-1)$.
Hence by the Laplace method, we can easily obtain the result.
\end{proof}

To extend the result to the case  $\theta\not=0$, $r\not=0$, we apply similar ideas and use the steepest descent method. 
\begin{lemma}
For $r\in(0,+\infty)$, $\theta\in(-\infty,+\infty)$, 
\[
f(y)=(\cosh^{-1} (\cosh r\cosh y))^2-(y-i\theta)^2
\]
defined on the strip $\{|\mathbf{Im}(y)|<\arccos\left( \frac{1}{\cosh r}\right)\}$ has a critical point at $i\varphi(r,\theta)$, where $\varphi(r,\theta)$ is the unique solution in $\left(-\arccos\left( \frac{1}{\cosh r}\right),\arccos\left( \frac{1}{\cosh r}\right)\right)$ to the equation
\begin{equation}\label{eq-varphi-r-theta}
\varphi(r,\theta)-\theta=\cosh r\sin\varphi(r,\theta)\frac{\cosh^{-1}(\cosh r\cos \varphi(r,\theta))}{\sqrt{\cosh^2r\cos^2\varphi(r,\theta)-1}}.
\end{equation}
\end{lemma}
\begin{proof}
Let $u=\cosh r\cos\varphi$. Notice that
\[
\frac{\partial}{\partial\varphi}\left(\varphi-\cosh r\sin\varphi\frac{\cosh^{-1}(\cosh r\cos\varphi)}{\sqrt{\cosh^2 r\cos^2\varphi-1}}\right)=\frac{\sinh^2 r}{u(r,\theta)^2-1}\left(1-\frac{u(r,\theta)\cosh^{-1} u(r,\theta)}{\sqrt{u^2(r,\theta)-1}}
\right).
\]
Since the right hand side of the above equation is negative,  we know that the map $\theta\rightarrow\varphi-\cosh r\sin\varphi\frac{\cos^{-1}(\cosh r\cos\varphi)}{\sqrt{\cosh^2 r\cos^2\varphi-1}}$ strictly decreases on $(-\infty, +\infty)$. Hence the solution is unique.
\end{proof}

\begin{prop}
Let  $r\in(0,+\infty)$, $\theta\in(-\infty,+\infty)$. When $t\to 0$,
\[
p_t(r,\theta)=\frac{1}{(4\pi t)^{n+\frac{1}{2}}\sinh r}\frac{(\cosh^{-1} u(r,\theta))^n}{\sqrt{\frac{u(r,\theta)\cosh^{-1} u(r,\theta)}{\sqrt{u^2(r,\theta)-1}}-1}}\frac{e^{-\frac{(\varphi(r,\theta)-\theta)^2\tanh^2 r}{4t\sin^2(\varphi(r,\theta))}}}{(u(r,\theta)^2-1)^{\frac{n-1}{2}}}(1+O(t)),
\]
where $u(r,\theta)=\cosh r\cos\varphi(r,\theta)$.
\end{prop}
\begin{proof}
From \eqref{eq8} and \eqref{eq9}, we know that
\[
p_t(r,\theta)\sim_{t\to0}\frac{1}{(4\pi t)^{n+1}}\int_{-\infty}^{+\infty}e^{-\frac{ f(y)}{4t} }g(y)dy,
\]
where $f(y)=(\cosh^{-1} (\cosh r\cosh y))^2-(y-i\theta)^2$, $g(y)=\left(\frac{\delta(y)}{\sinh\delta(y)}\right)^n$ and $\delta(y)=\cosh^{-1}(\cosh r\cosh y)$. Since $\delta(y)\ge\cosh^{-1}(\cosh r)=r>0$, we know that $e^{-\frac{ f(y)}{4t} }g(y)$ is analytic. Therefore for any $M>0$,
\[
\frac{1}{(4\pi t)^{n+1}}\int_{-\infty}^{+\infty}e^{-\frac{ f(y)}{4t} }g(y)dy=\frac{1}{(4\pi t)^{n+1}}\int_{-M}^{M}e^{-\frac{ f(y+i\varphi(r,\theta))}{4t} }g(y+i\varphi(r,\theta))dy+\ell_1-\ell_2,
\]
where $\ell_1=\frac{1}{(4\pi t)^{n+1}}\int_{0}^{\varphi(r,\theta)}e^{-\frac{ f(M+iy)}{4t} }g(M+iy)dy$ and $\ell_2=\frac{1}{(4\pi t)^{n+1}}\int_{0}^{\varphi(r,\theta)}e^{-\frac{ f(-M+iy)}{4t} }g(-M+iy)dy$. Simple calculations show that $\ell_1\to0$ and $\ell_2\to0$ as $M\to+\infty$.  Thus
\[
\frac{1}{(4\pi t)^{n+1}}\int_{-\infty}^{+\infty}e^{-\frac{ f(y)}{4t} }g(y)dy=\frac{1}{(4\pi t)^{n+1}}\int_{-\infty}^{+\infty}e^{-\frac{ f(y+i\varphi(r,\theta))}{4t} }g(y+i\varphi(r,\theta))dy.
\]
Moreover, at $i\varphi(r,\theta)$ we have
\begin{equation}\label{eq-f-dp}
f^{''}(i\varphi(r,\theta))=\frac{2\sinh^2 r}{u(r,\theta)^2-1}\left(\frac{u(r,\theta)\cosh^{-1} u(r,\theta)}{\sqrt{u^2(r,\theta)-1}}-1 \right),
\end{equation}
where $u(r,\theta)=\cosh r\cos\varphi(r,\theta)$. $f^{''}(i\varphi(r,\theta))$ is a real positive number. By the steepest descent method, we have that 
\begin{equation}\label{eq-pt-sim-sd}
\int_{-\infty}^{+\infty}e^{-\frac{ f(y)}{4t} }g(y)dy\sim_{t\to0}\sqrt{4\pi t}e^{-\frac {f(i\varphi(r,\theta))}{4t}}g(i\varphi(r,\theta))\sqrt{\frac{2}{f''(i\varphi(r,\theta))}}.
\end{equation}
By plugging in \eqref{eq-varphi-r-theta}, we have 
\begin{equation}\label{eq-f-i-varphi}
f(i\varphi(r,\theta))=(\cosh^{-1}(\cosh y\cos \varphi))^2-(\varphi-\theta)^2=\frac{(\varphi(r,\theta)-\theta)^2\tanh^2 r}{\sin^2(\varphi(r,\theta))},
\end{equation}
and since $\delta(i\varphi(r,\theta))=\cosh^{-1}(\cosh r\cos \varphi(r,\theta))=\cosh^{-1}u(r,\theta)$, we have
\begin{equation}\label{eq-g-i-varphi}
g(i\varphi(r,\theta))=\left(\frac{\delta(i\varphi(r,\theta))}{\sinh\delta(i\varphi(r,\theta))}\right)^n=\left(\frac{\cosh^{-1}u(r,\theta)}{\sqrt{u^2(r,\theta)-1}}\right)^n.
\end{equation}
By putting \eqref{eq-f-dp},  \eqref{eq-pt-sim-sd}, \eqref{eq-f-i-varphi} and \eqref{eq-g-i-varphi} together we have the conclusion. 
\end{proof}

At the end of this paper we present a by-product of the small time asymptotic estimations of the subelliptic heat kernels, that is,  the sub-Riemannian distance on $\bH^{2n+1}$ and $\widetilde{\bH^{2n+1}}$ by applying L\'eandre's result (see \cite{Leandre}). We remark it as below.
\begin{remark}
By radial symmetry, the sub-Riemannian distance from the north pole to any point on $\widetilde{\bH^{2n+1}}$ only depends on $r$ and $\theta$. If we denote it by $d(r,\theta)$, then from the previous propositions,
\item[(1)]
For $\theta\in \R$,
\[
d^2(0,\theta)=2\pi |\theta|+\theta^2;
\]
\item[(2)]
For $\theta\in \R$, $r\in\left(0,+\infty\right)$,
\[
d^2(r,\theta)=\frac{(\varphi(r,\theta)-\theta)^2\tanh^2 r}{\sin^2(\varphi(r,\theta))}.
\]
The above formulas also work for $\bH^{2n+1}$ if we restrict $\theta$ to $[-\pi,\pi]$.
\end{remark}
\textbf{Acknowledgements:} The author wishes to thank Prof. Fabrice Baudoin for  his constant interest and the many helpful conversations.


\begin{thebibliography}{10}
\bibitem{A} Agrachev, A., Boscain, U., Gauthier, J.P., Rossi, F. \textit{The intrinsic hypoelliptic Laplacian and its heat kernel on unimodular Lie groups}, Journal of Functional Analysis, Vol. 256, 8, (2009), 2621-2655.
\bibitem{B} Barilari, D., \textit{Trace heat kernel asymptotics in 3D contact sub-Riemannian geometry}, Journal of Mathematical Sciences, Volume 195, Issue 3, (2013) 391-411
\bibitem{BB} Baudoin, F., Bonnefont, M. \textit{The subelliptic heat kernel on $SU(2)$: representations, asymptotics and gradient bounds,} Math. Z. \textbf{263} (2009) 647-672
\bibitem{BW} Baudoin, F., Wang, J. \textit{The subelliptic heat kernel on the CR sphere}, 	Mathematische Zeitschrift, Volume 275, Issue 1-2, pp 135-150 (2013)
\bibitem{BGG} Beals, R., Gaveau, B., Greiner, P. C. \textit{Hamilton-Jacobi theory and the heat kernel on Heisenberg groups,} J. Math. Pures Appl. 79, 7 (2000) 633-689
\bibitem{BGS}Beals, R., Greiner, P. C., Stanton, N. K., \textit{The heat equation on a CR manifold}, J. Differential Geom. Volume 20, Number 2 (1984), 343-387

\bibitem{GBA1}Ben Arous, G., \emph{D\'eveloppement asymptotique du noyau de la chaleur hypoelliptique hors du cut-locus},
Ann. Sci. \'Ecole Norm. Sup. (4), 21 (1988), pp. 307-331.

\bibitem{GBA2}Ben Arous, G.,  L\'eandre, R., \emph{D\'ecroissance exponentielle du noyau de la chaleur sur la diagonale.
II}, Probab. Theory Related Fields, 90 (1991), pp. 377-402.

\bibitem{Bengtsson} Bengtsson, I., \textit{Anti-de Sitter Space}, Lecture notes, 1998

\bibitem{Bengtsson2} Bengtsson, I., Sandin P., \textit{Anti-de Sitter space, squashed and stretched} Classical Quantum Gravity, 23 (3) (2006), pp. 971-986
\bibitem{BFM} Branson, T., Fontana, L., Morpurgo, C. \textit{Moser-Trudinger and Beckner-Onofri's inequalities on the CR sphere,} Annals of Mathematics 177 (2013), 1-52
\bibitem{Bo} Bonnefont, M., \textit{The subelliptic heat kernel on SL(2,R) and on its universal covering:  integral representations and some functional inequalities}, Potential Analysis, (2012) Vol. 36, Number 2, 275-300

\bibitem{Ca} Carlip, S., \textit{Conformal field theory, $(2+1)$-dimensional gravity and the BTZ black hole} Classical Quantum Gravity, 22 (12) (2005), pp. R85-R123
\bibitem{CMV} Chang, D. C., Markina, I., Vasil'ev, A., \textit{Sub-Lorentzian geometry on anti-de Sitter space},  Journal de Math\'ematiques Pures et Appliqu\'ees
Volume 90, Issue 1, July 2008, Pages 82-110
\bibitem{CR} Dragomir, S., Tomassini, G., \textit{Differential geometry and analysis on CR manifolds}, Birkh\"auser, Vol. 246, 2006.
\bibitem{El} Eldredge, N., \textit{Gradient estimates for the subelliptic heat kernel on H-type groups}. J. Funct. Anal. 258 (2010), pp. 504-533. 
\bibitem{Go} Gadea, P. M., Oubi\~na, J. A. \textit{Homogeneous K\"ahler and Sasakian structures related to complex hyperbolic spaces}. Proc. Edinb. Math. Soc. (2) 53 (2010), no. 2, 393-413.
\bibitem{Ga} Gaveau B., \textit{Principe de moindre action, propagation de la chaleur et estim\'eees sous elliptiques sur certains groupes nilpotents}, Acta Math.
Volume 139, Number 1, 95-153, (1977).
\bibitem{Greiner} Greiner, P., \textit{A Hamiltonian Approach to the Heat Kernel of a SubLaplacian on S(2n+1)}, Analysis and Applications (2013), 11(6). 
\bibitem{Gr} Gruet J. C., \textit{Semi-groupe du mouvement brownien hyperbolique}, Stochastics and Stochastic Reprots. \textbf{56} (1996) 53-61

\bibitem{Leandre} L\'eandre, R., \textit{D\'eveloppement asymptotique de la densit\'e de diffusions d\'eg\'en\'er\'ees}, J. Probability Theorey and Related Fields. \textbf{76} (1987) 341-358

\bibitem{Leandre1} L\'eandre, R., \emph{Majoration en temps petit de la densit\'e d?une diffusion d\'eg\'en\'er\'ee}, Probab. Theory Related Fields, 74 (1987), no. 2, 289-294.

\bibitem{Leandre2}  L\'eandre, R., \emph{Minoration en temps petit de la densit\'e d?une diffusion d\'eg\'en\'er\'ee}, J. Funct. Anal. 74 (1987), no. 2, 399-414.

\bibitem{Li} Li, H.Q., \textit{Estimation optimale du gradient du semi-groupe de la chaleur sur le groupe de Heisenberg.}, Jour. Func. Anal., 236, pp 369-394, (2006).
\bibitem{LF} Lieb, E., Frank R., \textit{ Sharp constants in several inequalities on the Heisenberg group.}, Annals of Mathematics 176 (2012), 349-381
\bibitem{Na} Natrio, J., \textit{Relativity and singularities-a short introduction for mathematicians} Resenhas, 6 (4) (2005), pp. 309-335
\bibitem{magid}
Magid, M. A., \textit{Submersions from anti-de {S}itter space with totally geodesic fibers}, { J. Differential Geom.}, 16(2) pp. 323--331, 1981.
\bibitem{S-83}Strichartz, R. S., \textit{Analysis of the Laplacian on the complete Riemannian manifold}, Journal of Functional Analysis, Volume 52, Issue 1, (1983) pp 48-79
\bibitem{S}Strichartz, R. S., \textit{Sub-Riemannian geometry}, J. Diff. Geom. 24, no. 2, pp. 221-263 (1986)
\bibitem{T2}Taylor, M. E., \textit{Partial differential equations. II}, 2nd Edition, Applied Mathematical Sciences \textbf{116}, Springer-Verlag, New York (1996) 
\end{thebibliography}
\end{document}